\documentclass[letterpaper, 11 pt]{IEEEtran}
\IEEEoverridecommandlockouts
\usepackage{newcent}
\usepackage{amsmath}
\usepackage{amssymb}
\usepackage{psfrag,graphicx}
\usepackage{epsfig}
\usepackage{verbatim}
\usepackage{color}

\newtheorem{thm}{Theorem}

\newtheorem{prp}{Proposition}

\newtheorem{cor}{Corollary}

\newtheorem{rem}{Remark} 
 \newcommand{\trace}[1]{ \mathbf{Tr} \left( #1 \right) }    
 \newcommand{\E}[1]{ \mathbf{E} \left\{ #1 \right\} } 
 \newcommand{\bV}[1]{ \mathbf{V}} 
 \newcommand{\bX}[1]{ \mathbf{X}}
 \newcommand{\bR}[1]{ \mathbf{R}}
 \newcommand{\bU}[1]{ \mathbf{U}}
 \newcommand{\bW}[1]{ \mathbf{W}}

\title{\LARGE \bf Multi-Objective Optimal Control with Arbitrary Additive and Multiplicative Noise}
\author{Ather Gattami\\ Ericsson Research\\ Stockholm, Sweden\\ ather.gattami@ericsson.com
}
\begin{document}
\maketitle
\begin{abstract}
In this paper, we consider the problem of multi-objective optimal control of a dynamical system
with additive and multiplicative noises with given second moments and \textit{arbitrary} probability distributions.
The objectives are given by quadratic constraints in the state and controller, where the quadratic forms maybe 
indefinite and thus not necessarily convex. We show that the problem can be transformed to a semidefinite program 
and hence convex. The optimization problem is to be optimized with respect to a certain variable serving as the
covariance matrix of the state and the controller. We show that affine controllers are optimal and depend on the optimal covariance matrix. Furthermore, we show that optimal controllers are linear if all the quadratic forms
are convex in the control variable. 
The solutions are presented for both the finite and infinite horizon cases.
We give a necessary and sufficient condition for mean square stabilizability of the dynamical system with additive and multiplicative noises. The condition is a Lyapunov-like condition whose solution is again given by the covariance matrix of the state and the control variable. The results are illustrated with an example.
\end{abstract}

\section*{Notation}
\begin{tabular}{ll}
$\mathbb{S}^n$& The set of $n\times n$ symmetric matrices.\\
$\mathbb{S}^n_{+}$& The set of $n\times n$ symmetric positive\\
& semidefinite matrices.\\
$\mathbb{S}^n_{++}$& The set of $n\times n$ symmetric positive\\
& definite matrices.\\
$\succeq$& $A\succeq B$ $\Longleftrightarrow$ $A-B\in \mathbb{S}^n_{+}$.\\
$\succ$ 	& $A\succ B$ $\Longleftrightarrow$ $A-B\in \mathbb{S}^n_{++}$.\\
$\mathbf{Tr}$& $\mathbf{Tr}(A)$ is the trace of the matrix $A$.\\
$\delta(\cdot)$ & $\delta(0) = 1$, and $\delta(k) = 0$ for $k\neq 0$.\\
$\E \cdot $ & $\E x$ is the expected value\\ & of the random variable $x$.\\
& mean $m$ and covariance $X$.\\
$I_{n}$& Denotes the $n\times n$ identity matrix.\\
\end{tabular}

\section{Introduction}
\subsection{Background}
In this paper, we consider the problem of optimal control of a dynamical system with additive and
multiplicative noises of arbitrary distributions. More precisely, 
the optimal state-feedback control problem of discrete-time dynamical system with additive and multiplicative noises is given by

\begin{equation}
\label{introproblem}
\begin{aligned}
  \inf_{\mu_k}\hspace{2mm}   & \frac{1}{N}
  \sum_{k=0}^{N-1}    \E{z_k^T z_k}\\
\text{s. t.}\hspace{2mm}  x_{k+1} & = \left(A + \sum_{i=1}^M \sigma_k(i) A_i\right)x_k +  B u_k+ w_k \\
							z_k &= Cx_k + Du_k\\
                               u_k&=\mu_k(x_0,...,x_k)\\
\end{aligned}
\end{equation}
where $w_k$ and $\sigma_k(i)$ are independent white noises with unit variance and \textit{arbitrary} probability distribution functions. This is important since we would obtain a control problem whose solutions accounts for nonlinear uncertainty, which is modeled in the distribution of the stochastic noise.

The problem given by (\ref{introproblem}) is \textit{no longer a linear} quadratic control problem and the state is no longer Gaussian, not even if we assume that the noises are Gaussian. This nonlinear model is very important in many applications. 
For instance, multiplicative uncertainty appears on modeling link failures in networked control system (\cite{patterson2010convergence}). Also, this models uncertainty in the state space parameters of a linear dynamical system.

Optimal control of uncertain systems is one of the corner stones of control theory. There is a vast amount of research considering optimal control of linear dynamical systems with a quadratic objective, where the uncertainty is often modeled as additive noise with a Gaussian distribution. It's well known that the optimal state feedback control problem
\begin{equation*}
\begin{aligned}
  \min_{\mu_k}\hspace{2mm}   & \lim_{N\rightarrow \infty}\frac{1}{N}
  \sum_{k=0}^{N-1}    \E{z_k^T z_k}\\
\text{s. t.}\hspace{2mm}  x_{k+1} & = A x_k +  B u_k+ w_k \\
							z_k &= Cx_k + Du_k\\
                               u_k&=\mu_k(x_0,...,x_k)\\
\end{aligned}
\end{equation*}
has linear strategies $\mu_k$ as the optimal solution.
This yields a very nice structure on the optimization problem where an optimal linear controller can be obtained
by solving an optimization problem with polynomial time computational complexity. However, important variations of the linear quadratic optimal control problem are still not well understood. For instance, it's not clear whether linear controllers are optimal for the output feedback case under non Gaussian process noise assumptions and sub-optimality of linear controllers has become a folklore. Another variation of the linear quadratic control problem above is the problem given by (\ref{introproblem}). Although very similar to the linear quadratic control problem, the problem is nonlinear and the state is no longer Gaussian even if we assume Gaussian noises in the system. Therefore, linear controllers are not necessarily optimal.
 
This paper will address the following questions for the above optimal state-feedback control problem: 

\begin{itemize}
	\item[1)] For both a finite horizon $N$ and an infinite time-horizon, are linear controllers optimal  over the class of linear and nonlinear controllers?
	What is the optimal controller for arbitrary distributions of the noise variables? 
	\item[2)] What are the necessary and sufficient conditions for the optimal control problem to have a bounded solution as 
	$N\rightarrow \infty$ and that can be checked in polynomial time?
	\item[3)] What is the optimal controller if we add the (not necessarily convex) quadratic constraints
			$$
			\gamma_j(k) \geq 
                               \mathbf{E}\hspace{1mm}
                               \left[
  				\begin{matrix}
   			 		x_k\\
   				 	u_k
  				\end{matrix}
			\right]^TQ_j\left[
  				\begin{matrix}
    					x_k\\
    					u_k
  				\end{matrix}
			\right], ~~ j = 1, ..., J,
			$$
		for both finite and infinite time horizon $N$?\\
\end{itemize}

Stability of linear systems with additive noise is well studied. For a discrete-time system given by
$$
 x_{k+1}  = Ax_k + Bu_k + w_k,
$$
a necessary and sufficient condition for stabilizability is given by the Lyapunov matrix equation. 
The system above is stabelizable if and only if there exists a positive definite matrix $X$ such that
$$
X = AXA^T + BB^T.
$$
We are interested in finding a "Lyapunov-like" stabilizability criterion to answer our second question.

The last question is important since we may have practical energy constraints on the controller of the form
$$
 0<\gamma_1 \leq \E{u_k^Tu_k} \leq \gamma_2. 
$$
Another constraint could be the requirement of the energy of the controller to be proportional to the energy of the state:
$$
 \E{u_k^Tu_k} \leq c^2\E{x_k^Tx_k} 
$$
$$
\Updownarrow
$$
$$
 \E{u_k^Tu_k - c^2x_k^Tx_k} \leq 0.
 $$
None of the quadratic forms of the constraints above form is positive definite. This justifies the generality of the weights 
$Q_j$ to extend to the indefinite case, implying non convex optimal control problems.

\subsection{Previous Work}
Linear optimal state-feedback control for scalar discrete-time systems with multiplicative noise was considered in \cite{athans:1977} where the so called
\textit{uncertainty threshold principle} was introduced under Gaussian noise assumptions. It was shown that if 
the variance of the multiplicative noise exceeded a certain value, then stability (in the mean square sense) was not possible as the time
horizon tends to infinity. Further study for a special case of multivariate systems was considered in \cite{ku:1977}. 
The stabilization problem for discrete time systems with Gaussian noise was approached using linear matrix inequalities in \cite{li:2005} where
the class of linear controllers was considered. In \cite{primbs:2009}, the problem of finite horizon linear optimal control of discrete-time systems with 
multiplicative noise and convex quadratic constraints was approached using receding horizon control.

For continuous time stochastic systems where the noise is a standard Wiener process, stabilizability conditions in terms of linear matrix inequalities where considered in \cite{boyd:elghaoui:1994} and optimal control with respect to various norms was
considered in \cite{elghaoui:1995} by optimizing over linear controllers using linear matrix inequalities.

Linear optimal state-feedback control of linear systems with possibly indefinite quadratic constraints over an infinite time horizon was considered in \cite{yakubovich:1992} and its generalization to output-feedback control in  \cite{gattami:tac:10}, for both finite and infinite horizons, using a covariance formulation of the optimal control problem. It was also shown in \cite{gattami:tac:10} that linear controllers are optimal over the class of linear and nonlinear controllers.

\subsection{Contributions}
In this paper we answer the questions we posed in the Background section. In particular, we show that affine controllers are optimal and can be obtained by solving a semidefinite program. The solution depends solely on the the second moments of the
noises affecting the system, and thus, independent of the probability distributions of the noises. This shows that the Gaussian noise assumption used, even in classical linear quadratic control theory, is unnecessary.

We also show that linear controllers
are optimal if the quadratic constraints
$$
\gamma_j(k) \geq 
                               \mathbf{E}\hspace{1mm}
                               \left[
  				\begin{matrix}
   			 		x_k\\
   				 	u_k
  				\end{matrix}
			\right]^TQ_j\left[
  				\begin{matrix}
    					x_k\\
    					u_k
  				\end{matrix}
			\right], ~~ j = 1, ..., J,
$$
are convex in the control signal $u_k$(note that the quadratic constraints maybe convex in $u_k$ without being convex in both $x_k$ and $u_k$). The controllers are static, that is, they only depend on the current value of the state and are independent of 
the previously observed states. Furthermore, we show that the system is stabilizable in the mean square sense if and only if there exists a positive definite matrix $V\succeq 0$ such that
$$
FVF^T = 
\left[\begin{matrix} A & B\end{matrix}\right]
V
\left[\begin{matrix} A & B\end{matrix}\right]^T + \sum_{i=1}^M A_i FVF^T A_i^T + I,
$$
where 
$F = \begin{bmatrix} I_n & 0 \end{bmatrix}$ and $n$ is the state dimension. More specifically,
$V$ is the stationary covariance matrix of the vector consisting of the state and the controller,
$$
V = 
\mathbf{E} 
	\left[ \begin{matrix}
    					x_k\\
    					u_k
  				\end{matrix}
			\right]
			\left[
			\begin{matrix}
    					x_k\\
    					u_k
  				\end{matrix}
			\right]^T.
$$



\section{Optimal State Feedback}
Consider a dynamical system given by 
\begin{equation}
	\label{dsys}
		\begin{aligned}
			x_{k+1} &= \left(A + \sum_{i=1}^M \sigma_k(i) A_i\right)x_k + B u_k + w_k\\
				z_k &= Cx_k + Du_k
		\end{aligned},~ x_0 = 0,
\end{equation}
where $\{w_k\}$ and $\{\sigma_k(i)\}$ are independent zero-mean white noises with \textit{arbitrary} 
probability distributions. Without loss of generality, we will assume that 
$\E{w_kw_k^T} = I$ and $\E {\sigma^2_k(i)} = 1$ for all $i$ and $k$. If $\E {\sigma^2_k(i)} = s_k^2(i)$, then we may
replace $A_i$ with $\bar{A}_i = s_k(i) \cdot A_i $ and again solve for the case $\E {\sigma^2_k(i)} = 1$.
We want to find the optimal causal state feedback controller 
$$u_k = \mu_k(x_0, ..., x_k)$$ 
that minimizes the average variance for the sequence $\{z_k\}$:
$$
 \frac{1}{N}\sum_{k=0}^{N-1} \E{z_k^T z_k} \rightarrow \min.
$$
Thus, we want to solve the optimization problem
\begin{equation}
\label{lqproblem}
\begin{aligned}
  \min_{\mu_k}\hspace{2mm}   & \frac{1}{N}
  \sum_{k=0}^{N-1}    \E{z_k^T z_k}\\
\text{s. t.}\hspace{2mm}  x_{k+1} & = \left(A + \sum_{i=1}^M \sigma_k(i) A_i\right)x_k +  B u_k+ w_k \\
							z_k &= Cx_k + Du_k\\
                              \E{\sigma_k(i) w_\ell}&=0, ~\forall i, k, \ell\\
                               \E {w_kw^T_\ell}&=\delta(k-\ell)\cdot I\\
                               \E {\sigma_k(i) \sigma_\ell(i')} &= \delta(i-i')\delta(k-\ell)\\
                               u_k&=\mu_k(x_0,...,x_k)
\end{aligned}
\end{equation}
where $x_k\in \mathbb{R}^n$ and  $u_k\in \mathbb{R}^m$.\\

For simplicity, we will assume that the following matrix is positive definite:
$$
Q = \left[\begin{matrix}
C & D
\end{matrix}\right]^T
\left[\begin{matrix}
C & D
\end{matrix}\right] \succ 0.
$$
Introduce the positive semidefinite matrix
\begin{equation*}
V_k =\left[\begin{matrix}
X_k & R_k\\
R^T_k & U_k
\end{matrix}\right]=\mathbf{E}\left\{\hspace{1mm}
  \left[
    \begin{matrix}
      x_k\\
      u_k
    \end{matrix}
  \right]
  \left[
    \begin{matrix}
      x_k\\
      u_k
    \end{matrix}
  \right]^T\right\}.
\end{equation*}

The system dynamics and the correlation assumptions in (\ref{lqproblem}) implicate the
following recursive equation for the covariance matrices
$\{V_k\}$:

\begin{equation}
\label{Xk}
\begin{aligned}
&X_{k+1}\\ 
&= \E{x_{k+1}x^T_{k+1}}\\
&= \mathbf{E}\hspace{1mm} \left\{\left(Ax_k + \sum_{i=1}^M \sigma_k(i) A_i x_k +  B u_k + w_k\right)\right. \times \\
&\hspace{1.1cm} \left. \left(Ax_k + \sum_{i=1}^M \sigma_k(i) A_i x_k +  B u_k+ w_k\right)^T\right\}\\
&=\mathbf{E}\hspace{1mm}\left\{\left[\begin{matrix} A & B\end{matrix}\right]
\left[
    \begin{matrix}
      x_k\\
      u_k
    \end{matrix}
  \right] 
\left[
    \begin{matrix}
      x_k\\
      u_k
    \end{matrix}
  \right]^T 
\left[\begin{matrix} A & B\end{matrix}\right]^T \right \}\\
& + \mathbf{E}\hspace{1mm}\left\{ \sum_{i=1}^M \sigma^2_k(i) A_i x_k x_k^T A_i^T\right\} + \E{w_kw_k^T} \\
&= \left[\begin{matrix} A & B\end{matrix}\right]
V_k
\left[\begin{matrix} A & B\end{matrix}\right]^T + \sum_{i=1}^M A_i X_k A_i^T + I, \\
\end{aligned}
\end{equation}
where we used that $\E{w_k x_k^T}=0$, $\E{w_k u_k^T}=0$, 
$\sigma_k = (\sigma_k(1), ..., \sigma_k(M))$, $\E{w_k \sigma_k^T}=0$,
$\E{x_k \sigma_k^T}=0$, $\E{u_k \sigma_k^T}=0$.  \\

\begin{thm}
\label{stat}
The optimal value of the optimization problem (\ref{lqproblem}) is equal to the optimal value of the following semidefinite program:

\begin{equation}
	\label{rec}
	\begin{aligned}
		\min_{V_k\succeq 0}\hspace{2mm} &  \frac{1}{N}  \sum_{k=0}^{N-1} 
		\trace{\left[\begin{matrix}
C & D
\end{matrix}\right] V_k
\left[\begin{matrix}
C & D
\end{matrix}\right]^T}\\
		\text{s. t.}\hspace{2mm} 	& FV_{k+1}F^T = \left[\begin{matrix} A & B \end{matrix}\right]
												V_k\left[\begin{matrix} A & B\end{matrix}\right]^T \\
									& \hspace{1.6cm} + \sum_{i=1}^M A_i FV_kF^T A_i^T + I \\
	\end{aligned}
\end{equation}
where 
$F = \begin{bmatrix} I_n & 0 \end{bmatrix}$.\\
\end{thm}

\begin{proof} The constraints in (\ref{lqproblem}) give rise to the equality constraint (\ref{Xk}).
For the left hand side of Equation (\ref{Xk}), we have that
$$
X_{k+1} = FV_{k+1} F^T,
$$
which gives the equality constraint in (\ref{rec}).
Also, we have that
\begin{equation*}
	\begin{aligned}
		z_k^T  z_k 	&= \trace{z_k  z_k^T} \\
					&= \trace{
							\left[\begin{matrix} C & D\end{matrix}\right]
								\left[  \begin{matrix}
      										x_k\\
      										u_k
    								\end{matrix}\right]
 								 \left[ \begin{matrix}
      										x_k\\
     										 u_k
    								\end{matrix}\right]^T
							\left[\begin{matrix} C & D\end{matrix}\right]^T 	
						}.
	\end{aligned}
\end{equation*}
Hence,
\begin{equation*}
	\begin{aligned}
&\E{z_k^T z_k} \\	
		&= \E {\trace{
							\left[\begin{matrix} C & D\end{matrix}\right]
								\left[  \begin{matrix}
      										x_k\\
      										u_k
    								\end{matrix}\right]
 								 \left[ \begin{matrix}
      										x_k\\
     										 u_k
    								\end{matrix}\right]^T
							\left[\begin{matrix} C & D\end{matrix}\right]^T 	
						}
				}\\
			&=\trace{\E{
							\left[\begin{matrix} C & D\end{matrix}\right]
								\left[  \begin{matrix}
      										x_k\\
      										u_k
    								\end{matrix}\right]
 								 \left[ \begin{matrix}
      										x_k\\
     										 u_k
    								\end{matrix}\right]^T
							\left[\begin{matrix} C & D\end{matrix}\right]^T 	
						}
				}\\
			& = \trace{
						\left[\begin{matrix} C & D\end{matrix}\right]
												V_k \left[\begin{matrix} C & D\end{matrix}\right]^T 
						}.
	\end{aligned}
\end{equation*}

Thus, the objective to be minimized is given by
{\small
$$
 \frac{1}{N}\sum_{k=0}^{N-1} \E{z_k^T z_k} = 
 \frac{1}{N}\sum_{k=0}^{N-1}    	\trace{
						\left[\begin{matrix} C & D\end{matrix}\right]
												V_k \left[\begin{matrix} C & D\end{matrix}\right]^T 
						},
$$
}
and we are done.\\
\end{proof}

Before proceeding to the next result, we need the following proposition.\\

\begin{prp}
\label{v}
Consider  the dynamical systems given by  
$$
x_{k+1} = \left(A + \sum_{i=1}^M \sigma_k(i) A_i\right)x_k + B u_k + Bv_k + w_k,
$$
$$
x'_{k+1} = \left(A + \sum_{i=1}^M \sigma_k(i) A_i\right)x'_k + B u_k + w_k,
$$
$$
x''_{k+1} = \left(A + \sum_{i=1}^M \sigma_k(i) A_i\right)x''_k + Bv_k
$$
Let
$v_k$ be uncorrelated with $w_l$ and $u_l$ for all $k$ and $l$. Also, let 
$$\E{v_kv_k^T} = \Pi_k,$$ 
and
$$
\E{w_k w_k^T} = W_k,
$$
for $k=0, ..., N$. Then, $x_k = x'_k + x''_k$ and
$$
V_k = 
\mathbf{E} 
	\left[ \begin{matrix}
    					x_k\\
    					u_k
  				\end{matrix}
			\right]
			\left[
			\begin{matrix}
    					x_k\\
    					u_k
  				\end{matrix}
			\right]^T
 \succeq 
 \mathbf{E} 
	\left[ \begin{matrix}
    					x'_k\\
    					u_k
  				\end{matrix}
			\right]
			\left[
			\begin{matrix}
    					x'_k\\
    					u_k
  				\end{matrix}
			\right]^T
 = V'_k.
$$
\end{prp}

\begin{proof}
The relation $x_k = x'_k + x''_k$ is easy to verify by summing up the right and left hand  sides of the
systems equations above for $x'_k$  and $x''_k$ . Now since $v_k$ is uncorrelated with $w_l$ and $u_l$ for all $k$ and $l$, 
we have that $\E{ x'_k (x''_k)^T} = 0$, $\E {u_k (x''_k)^T} = 0$, and
\begin{align*}
V_k &= 
\mathbf{E} 
	\left[ \begin{matrix}
    					x_k\\
    					u_k
  				\end{matrix}
			\right]
			\left[
			\begin{matrix}
    					x_k\\
    					u_k
  				\end{matrix}
			\right]^T\\
	&= 
 \mathbf{E} 
	\left[ \begin{matrix}
    					x'_k + x''_k\\
    					u_k
  				\end{matrix}
			\right]
			\left[
			\begin{matrix}
    					x'_k + x''_k\\
    					u_k
  				\end{matrix}
			\right]^T	\\
	&=
	\mathbf{E} 
	\left[ \begin{matrix}
    					x'_k\\
    					u_k
  				\end{matrix}
			\right]
			\left[
			\begin{matrix}
    					x'_k\\
    					u_k
  				\end{matrix}
			\right]^T + \E {x''_k(x''_k)^T}\\
	& \succeq 
 \mathbf{E} 
	\left[ \begin{matrix}
    					x'_k\\
    					u_k
  				\end{matrix}
			\right]
			\left[
			\begin{matrix}
    					x'_k\\
    					u_k
  				\end{matrix}
			\right]^T 
 = V'_k.
\end{align*} 
\end{proof}

\begin{rem}
The above proposition clearly shows that the controller doesn't benefit from adding additional noise, 
$u_k \mapsto u_k + v_k$,
if the objective to be minimized is increasing in the covariance of the state and controller.  
\end{rem}

The next result shows how to obtain the optimal controller from the optimal covariance matrices.

\begin{thm}
\label{finitehorizon}
Let  
$$
V_k^\star = \left[\begin{matrix}
			X_k^\star & R_k^\star\\
			(R_k^\star)^T & U_k^\star
		\end{matrix}\right]
$$
be a solution to

\begin{equation}
	\label{optcon}
	\begin{aligned}
		\min_{V_k\succeq 0}\hspace{2mm} & \frac{1}{N}\sum_{k=0}^N
		\trace{\left[\begin{matrix}
C & D
\end{matrix}\right] V_k
\left[\begin{matrix}
C & D
\end{matrix}\right]^T}\\
		\text{s. t.}\hspace{2mm} 	& FV_{k+1}F^T = \left[\begin{matrix} A & B\end{matrix}\right]
												V_k\left[\begin{matrix} A & B\end{matrix}\right]^T \\
									& \hspace{1.65cm} + \sum_{i=1}^M A_i FV_kF^T A_i^T + I \\
	\end{aligned}
\end{equation}
with $F = \begin{bmatrix} I_n & 0 \end{bmatrix}$. Then, an optimal solution to the optimization problem (\ref{lqproblem}) 
is given by
$$
u_k = (R_k^\star)^T (X_k^\star)^{-1} x_k.\\
$$
\end{thm}

\begin{proof}
Since $V_k^\star$ is positive semidefinite, the Schur complement of $X^\star$ in $V^\star$ is also positive semidefinite:
$$
U_k^\star - (R_k^\star)^T (X_k^\star)^{-1} (R_k^\star)=: \Pi_k^\star \succeq 0.
$$
Consider the controller
$$
u_k = (R_k^\star)^T (X_k^\star)^{-1} x_k + v_k,
$$
where 
$$\E{v_kv_k^T} =  \Pi_k^\star$$ 
and $v_k$ is independent of $\sigma_l(i)$, $x_l$, and $w_l$, for all $i$, $k$, and $l$. 
It's easy to verify that the controller
gives the covariance sequence $\{V_k^\star\}$, and thus achieves the minimal solution
of (\ref{lqproblem}) according to Theorem \ref{stat}:
\begin{align*} 
	&\mathbf{E} 
		\left[ \begin{matrix}
    					x_k\\
    					(R_k^\star)^T (X_k^\star)^{-1} x_k + v_k
  				\end{matrix}
			\right]
			\left[
			\begin{matrix}
    					x_k\\
    					(R_k^\star)^T (X_k^\star)^{-1} x_k  +v_k
  				\end{matrix}
			\right]^T\\
	&= 
		\left[ \begin{matrix}
    					X_k^\star 		& R_k^\star\\
    					(R_k^\star)^T & (R_k^\star)^T (X_k^\star)^{-1}  R_k^\star + \Pi_k^\star
  				\end{matrix}
			\right]\\
	&= V_k^\star.
\end{align*}
 But $v_k$ is just additional noise and by removing it from the controller 
we will not increase the objetive value according to Proposition \ref{v}. Thus, an optimal controller is given by  
$$
u_k = (R_k^\star)^T (X_k^\star)^{-1} x_k,
$$
and the proof is complete.\\
\end{proof}

\begin{thm}
\label{stat2}
The optimal value of  the optimization 
problem (\ref{lqproblem}) as $N\rightarrow \infty$ is equal to the value of the optimization problem

\begin{equation}
	\label{optcon}
	\begin{aligned}
		\min_{V\succeq 0}\hspace{2mm} & \trace{\left[\begin{matrix}
C & D
\end{matrix}\right] V
\left[\begin{matrix}
C & D
\end{matrix}\right]^T}\\
		\text{s. t.}\hspace{2mm} 	& FVF^T = \left[\begin{matrix} A & B\end{matrix}\right]
												V\left[\begin{matrix} A & B\end{matrix}\right]^T \\
									& \hspace{1.1cm} + \sum_{i=1}^M A_i FVF^T A_i^T + I \\
	\end{aligned}
\end{equation}
with $F = \begin{bmatrix} I_n & 0 \end{bmatrix}$. Furthermore, if
$$
V^\star = \left[\begin{matrix}
			X^\star & R^\star\\
			(R^\star)^T & U^\star
		\end{matrix}\right]
$$
is an optimal solution to (\ref{optcon}), then $X^\star\succ 0$ and an optimal solution to the optimization 
problem (\ref{lqproblem}) is given by
$$
u_k = (R^\star)^T (X^\star)^{-1} x_k
$$
as $N\rightarrow \infty$.\\
\end{thm}

\begin{proof} Theorem \ref{stat} and the assumption that
$$
\left[\begin{matrix}
C & D
\end{matrix}\right]^T
\left[\begin{matrix}
C & D
\end{matrix}\right] \succ 0
$$
implies that the solution of the optimization problem (\ref{lqproblem}) is bounded if and only if
$$
\lim_{N\rightarrow \infty} \frac{1}{N} \sum_{k=0}^{N-1} V_k
$$
is bounded according to 
Thus, the value of (\ref{optcon}) is finite if and only if
the value of the optimization problem (\ref{lqproblem}) is finite.
The left hand side of Equation (\ref{rec}) has the asymptotic average
$FVF^T$. The right hand side of Equation (\ref{rec}) has the asymptotic average
$$
\left[\begin{matrix} A & B\end{matrix}\right]
V
\left[\begin{matrix} A & B\end{matrix}\right]^T + \sum_{i=1}^M A_i FVF^T A_i^T + I, 
$$
and we get the equality constraint in (\ref{optcon}).
Also, the objective to be minimized is given by
$$
 \lim_{N\rightarrow \infty}\frac{1}{N}\sum_{k=0}^{N-1} \E{z_k^T z_k} = \trace{
						\left[\begin{matrix} C & D\end{matrix}\right]
												V \left[\begin{matrix} C & D\end{matrix}\right]^T 
						}.
$$
According to Theorem \ref{finitehorizon}, an optimal sequence $\{V_k^\star\}$ 
is obtained for the controller $u_k = (R_k^\star)^T (X_k^\star)^{-1} x_k$. Thus,  
$$
 \frac{1}{N}\sum_{k=0}^{N-1} V_k^\star \rightarrow V^\star 
$$
as $N\rightarrow \infty$, which implies that
$$
(R_k^\star)^T (X_k^\star)^{-1}  \rightarrow  (R^\star)^T (X^\star)^{-1} 
$$
as $N\rightarrow \infty$. This completes the proof.\\
\end{proof}
{
As a corollary, we get a Lyapunov-like necessary and sufficient condition for mean square stabilizability 
of the stochastic system.\\
 }
 
\begin{cor}
The system 
$$
x_{k+1}  = \left(A + \sum_{i=1}^M \sigma_k(i) A_i\right)x_k +  B u_k+ w_k 
$$
is stabilizable if and only if there exists a matrix $V\succeq 0$ such that
$$
FVF^T =
\left[\begin{matrix} A & B\end{matrix}\right]
V
\left[\begin{matrix} A & B\end{matrix}\right]^T + \sum_{i=1}^M A_i FVF^T A_i^T + I,
$$
where 
$F = \begin{bmatrix} I_n & 0 \end{bmatrix}$.
\end{cor}

\section{Optimal State Feedback Control with Arbitrary Quadratic Constraints}
\label{constrainedlq}
In this section, we consider a linear quadratic problem given
by (\ref{lqproblem}), with additional constraints of the form
$$
\mathbf{E}\hspace{1mm}\left[
  \begin{matrix}
    x_k\\
    u_k
  \end{matrix}
\right]^TQ_j \left[
  \begin{matrix}
    x_k\\
    u_k
  \end{matrix}
\right]\leq \gamma_j(k).
$$

We do \textit{not} make any other assumptions
about $Q_j$ except that it is symmetric, $Q_j\in \mathbb{S}^{m+n}$.

Thus, the constrained optimal control problem we are considering in this section is given by
\begin{equation}
\label{constrainedlqproblem}
\begin{aligned}
  \min_{\mu_k}\hspace{2mm}   & \frac{1}{N}
  \sum_{k=0}^{N-1}    \E{z_k^T z_k}\\
\text{s. t.}\hspace{2mm}  x_{k+1} & = \left(A + \sum_{i=1}^M \sigma_k(i) A_i\right)x_k +  B u_k+ w_k \\
							z_k &= Cx_k + Du_k\\
                              \E{\sigma_k(i) w_\ell}&=0, ~\forall i, k, \ell\\
                               \E {w_kw^T_\ell}&=\delta(k-\ell)\cdot I\\
                               \E {\sigma_k(i) \sigma_\ell(i')} &= \delta(i-i')\delta(k-\ell)\\
                               u_k&=\mu_k(x_0,...,x_k)\\
                               \gamma_j(k) &\geq 
                               \mathbf{E}\hspace{1mm}
                               \left[
  				\begin{matrix}
   			 		x_k\\
   				 	u_k
  				\end{matrix}
			\right]^TQ_j\left[
  				\begin{matrix}
    					x_k\\
    					u_k
  				\end{matrix}
			\right], ~~ j = 1, ..., J
\end{aligned}
\end{equation}

\begin{thm}
\label{constraints}
	The optimal control problem (\ref{constrainedlqproblem})
	is equivalent to  the following semidefinite program:

\begin{equation}
	\label{constrainrec}
	\begin{aligned}
		\min_{V_k\succeq 0}\hspace{2mm} &  \frac{1}{N}  \sum_{k=0}^{N-1} 
		\trace{\left[\begin{matrix}
C & D
\end{matrix}\right] V_k
\left[\begin{matrix}
C & D
\end{matrix}\right]^T}\\
		\text{s. t.}\hspace{2mm} 	& FV_{k+1}F^T = \left[\begin{matrix} A & B \end{matrix}\right]
												V_k\left[\begin{matrix} A & B\end{matrix}\right]^T \\
									& \hspace{1.65cm} + \sum_{i=1}^M A_i FV_kF^T A_i^T + I \\
									& \hspace{8mm} \gamma_j(k) \geq \mathbf{Tr}\hspace{1mm}Q_j V_k, ~~ j=1,...,J.&
	\end{aligned}
\end{equation}
where 
$F = \begin{bmatrix} I_n & 0 \end{bmatrix}$.\\
\end{thm}
\begin{proof}
Note first that 
$$
\left[
  			\begin{matrix}
   			 	x_k\\
   				 u_k
  			\end{matrix}
		\right]^TQ_j \left[
  			\begin{matrix}
    				x_k\\
    				u_k
  			\end{matrix}
		\right]
$$
is a scalar, so 
$$
		\left[
  			\begin{matrix}
   			 	x_k\\
   				 u_k
  			\end{matrix}
		\right]^TQ_j \left[
  			\begin{matrix}
    				x_k\\
    				u_k
  			\end{matrix}
		\right] = 
		\trace{\hspace{1mm}\left[
  			\begin{matrix}
   			 	x_k\\
   				 u_k
  			\end{matrix}
		\right]^TQ_j \left[
  			\begin{matrix}
    				x_k\\
    				u_k
  			\end{matrix}
		\right]}.
$$
Thus,
\begin{equation}
\label{covtraceder}
	\begin{aligned}
		\gamma_j(k) &\geq 
		\mathbf{E}\hspace{1mm}\left[
  			\begin{matrix}
   			 	x_k\\
   				 u_k
  			\end{matrix}
		\right]^TQ_j \left[
  			\begin{matrix}
    				x_k\\
    				u_k
  			\end{matrix}
		\right]\\
		&=
		\mathbf{E}\hspace{1mm} \trace{\left[
  			\begin{matrix}
   			 	x_k\\
   				 u_k
  			\end{matrix}
		\right]^TQ_j \left[
  			\begin{matrix}
    				x_k\\
    				u_k
  			\end{matrix}
		\right]}\\
		&=
		\mathbf{E} \hspace{1mm}\trace{Q_j\left[
  			\begin{matrix}
    				x_k\\
    				u_k
  			\end{matrix}
		\right]
		\left[
  			\begin{matrix}
   			 	x_k\\
   				 u_k
  			\end{matrix}
		\right]^T}\\
		&= \trace{Q_j V_k}.
	\end{aligned}
\end{equation}
Hence, we may write the quadratic constraints as the linear constraint
\begin{equation}
\label{covtrace}
\mathbf{Tr}\hspace{1mm}Q_j V_k\leq \gamma_j(k),
\end{equation}
for $k=0,...,N-1$, $j=1,...,J$. The rest of the proof follows the same lines as Theorem \ref{stat}.\\
\end{proof}

\begin{rem}
Note that the covariance constraints 
$$\mathbf{Tr}\hspace{1mm}Q_j V_k\leq \gamma_j(k)$$
are linear in the elements of the covariance matrices $\{V_k\}$, and hence
convex. This shows that the covariance formulation turns the original non convex optimal control problem into a convex one.\\
\end{rem}

\begin{thm}
\label{constrainedfinitehorizon}
Let  
$$
V_k^\star = \left[\begin{matrix}
			X_k^\star & R_k^\star\\
			(R_k^\star)^T & U_k^\star
		\end{matrix}\right]
$$
be a solution to

\begin{equation}
	\label{constrainedoptcon}
	\begin{aligned}
		\min_{V_k\succeq 0}\hspace{2mm} & \frac{1}{N}\sum_{k=0}^N
		\trace{\left[\begin{matrix}
C & D
\end{matrix}\right] V_k
\left[\begin{matrix}
C & D
\end{matrix}\right]^T}\\
		\text{s. t.}\hspace{2mm} 	& FV_{k+1}F^T = \left[\begin{matrix} A & B\end{matrix}\right]
												V_k\left[\begin{matrix} A & B\end{matrix}\right]^T \\
									& \hspace{1.65cm} + \sum_{i=1}^M A_i FV_kF^T A_i^T + I \\																		& \hspace{8mm} \gamma_j(k) \geq \mathbf{Tr}\hspace{1mm}Q_j V_k, ~~ j=1,...,J.&\\
	\end{aligned}
\end{equation}
with $F = \begin{bmatrix} I_n & 0 \end{bmatrix}$. Then, an optimal solution to the optimization problem (\ref{constrainedlqproblem}) 
is given by
$$
u_k = (R_k^\star)^T (X_k^\star)^{-1} x_k + v_k.
$$
with 
$$\E{v_kv_k^T} =  \Pi_k^\star$$ 
and 
$$
\Pi_k^\star = U_k^\star - (R_k^\star)^T (X_k^\star)^{-1} (R_k^\star).
$$
\end{thm}

\begin{proof}
Since $V_k^\star$ is positive semidefinite, the Schur complement of $X^\star$ in $V^\star$ is nonnegative
$$
\Pi_k^\star = U_k^\star - (R_k^\star)^T (X_k^\star)^{-1} (R_k^\star)\succeq 0.
$$
Consider the controller
$$
u_k = (R_k^\star)^T (X_k^\star)^{-1} x_k + v_k,
$$
where 
$$\E{v_kv_k^T} =  \Pi_k^\star$$ 
and $v_k$ is independent of $\sigma_l(i)$ and $w_l$, for all $i$, $k$, and $l$. 
It's easy to verify that the controller gives the covariance sequence $\{V_k^\star\}$, and thus achieves the minimal solution
of (\ref{constrainedlqproblem}) according to Theorem \ref{constraints}. This completes the proof.\\
\end{proof}

\begin{rem}
Note that the structure of the solution to the multi-objective optimal control problem (\ref{constrainedfinitehorizon})
is a little bit different from that of the one with no constraints (\ref{finitehorizon}) because of the additional term $v_k$.
One interpretation to this structure is that if we had the non convex constraint $\E {u_k^Tu_k} \geq \gamma >0$, then
the controller $u_k$ adds an offset $v_k$ in order to be above a certain energy level, although it does neither contribute to
stabilization nor optimization of a convex quadratic objective. \\

Now suppose that all the quadratic constraints 
$$
\mathbf{E}\hspace{1mm}\left[
  \begin{matrix}
    x_k\\
    u_k
  \end{matrix}
\right]^TQ_j \left[
  \begin{matrix}
    x_k\\
    u_k
  \end{matrix}
\right]\leq \gamma_j(k)
$$
are convex in $u_k$ (note that 
they don't need to be convex in $x_k$), then an optimal controller would still be such that $v_k=0$.
To see this, let 
$$
Q_j = 
\left[
  \begin{matrix}
    E_j & G_j\\
    G_j^T & H_j
  \end{matrix}
\right]
$$
with $H_j\succeq 0$ in order for the quadratic form to be convex in $u_k$. Then, 
$$\gamma_j(k) \geq \mathbf{Tr}\hspace{1mm}Q_j V_k = \mathbf{Tr}\hspace{1mm}H_j U_k + \cdots$$
Now consider the controller $u_k = R_k^T X_k^{-1} x_k + v_k$ which has the covariance
$$
U_k = \Pi_k + R_k^T X_k^{-1} R_k,
$$
where $v_k$ is uncorrelated with $x_k$ and $\mathbf{E}\hspace{1mm} v_kv_k^T = \Pi_k$.
Removing $v_k$ implies that the covariance of $u_k$ is given by
$$
U_k = R_k^T X_k^{-1} R_k,
$$
which decreases the value of $\mathbf{Tr}\hspace{1mm}Q_j V_k $ with the amount 
$$\mathbf{Tr}\hspace{1mm}H_j \Pi_k \geq 0,$$ 
and the inequality $\gamma_j(k) \geq\mathbf{Tr}\hspace{1mm}Q_j V_k$ would still hold.\\
\end{rem}

Finally, we state the result for the infinite horizon case.\\

\begin{thm}
\label{constraints3}
The optimal value of  the optimization 
problem (\ref{lqproblem}) as $N\rightarrow \infty$ is equal to the value of the optimization problem

\begin{equation}
	\label{optcon2}
	\begin{aligned}
		\min_{V\succeq 0}\hspace{2mm} & \trace{\left[\begin{matrix}
C & D
\end{matrix}\right] V
\left[\begin{matrix}
C & D
\end{matrix}\right]^T}\\
		\text{s. t.}\hspace{2mm} 	& FVF^T = \left[\begin{matrix} A & B\end{matrix}\right]
												V\left[\begin{matrix} A & B\end{matrix}\right]^T \\
									& \hspace{1.1cm} + \sum_{i=1}^M A_i FVF^T A_i^T + I \\
									& \hspace{8mm} \gamma_j \geq \mathbf{Tr}\hspace{1mm}Q_j V, ~~ j=1,...,J.&\\
	\end{aligned}
\end{equation}
with $F = \begin{bmatrix} I_n & 0 \end{bmatrix}$. Furthermore, if
$$
V^\star = \left[\begin{matrix}
			X^\star & R^\star\\
			(R^\star)^T & U^\star
		\end{matrix}\right]
$$
is an optimal solution to (\ref{optcon}), then $X^\star\succ 0$ and an optimal solution to the optimization 
problem (\ref{lqproblem}) is given by
$$
u_k = (R^\star)^T (X^\star)^{-1} x_k + v_k,
$$
$$\E{v_kv_k^T} =  U^\star - (R^\star)^T (X^\star)^{-1} (R^\star),$$
as $N\rightarrow \infty$.\\
\end{thm}

\begin{proof}
The proof follows from Theorem \ref{constrainedfinitehorizon} using similar arguments as in Theorem \ref{stat2}, and therefore omitted here.
\end{proof}


\section{Numerical Example}
Consider the state feedback problem
\begin{equation}
\begin{aligned}
  \min_{\mu_k}\hspace{2mm}   & \lim_{N \rightarrow \infty } \frac{1}{N}
  \sum_{k=0}^{N-1}    \E{z_k^T z_k}\\
\text{s. t.}\hspace{2mm}  x_{k+1} & = \left(A + \sigma_k A_1\right)x_k +  B u_k+ w_k \\
							z_k &= Cx_k + Du_k\\
                              \E{\sigma_k(i) w_\ell}&=0, ~\forall i, k, \ell\\
                               \E {w_kw^T_\ell}&=\delta(k-\ell)\cdot I\\
                               \E {\sigma_k(i) \sigma_\ell(i')} &= \delta(i-i')\delta(k-\ell)\\
                               u_k&=\mu_k(x_0,...,x_k)\\
                               0 &\geq 
                               \mathbf{E}\hspace{1mm}
                               \left[
  				\begin{matrix}
   			 		x_k\\
   				 	u_k
  				\end{matrix}
			\right]^TQ\left[
  				\begin{matrix}
    					x_k\\
    					u_k
  				\end{matrix}
			\right]
\end{aligned}
\end{equation}
with 
\begin{align*}
A &=
\begin{bmatrix}
	1  &   2\\
     	4  &   1
\end{bmatrix}\\
A_1 &=
\begin{bmatrix}
	0.5 &   0\\
     	0    &   0.5
\end{bmatrix}\\
B &=
\begin{bmatrix}
	1\\
     	1
\end{bmatrix}\\
C &=
\begin{bmatrix}	
     1  &   0\\
     0  &  1\\
     0  &   0
\end{bmatrix}\\
D &=
\begin{bmatrix}
	0 \\
	0 \\
     	1 
\end{bmatrix}\\
Q &=
\begin{bmatrix}
	-4 & 0 & 0 \\
	0 & -4 & 0 \\
     	 0 & 0 & 1 
\end{bmatrix}.
\end{align*}
Solving the semidefinite program (\ref{optcon2}) gives the solution
$$
V^\star=
\begin{bmatrix}		
    58.8   & 131.2 &  -283.7\\
    131.2  &  309.5  & -674.8\\
   283.7   & -674.8  &  1473.1
\end{bmatrix}.
$$
Theorem \ref{constraints3} gives that the optimal controller is given by
$$
u_k = (R^\star)^T (X^\star)^{-1} x_k + v_k,
$$
with
$$
\E{v_kv_k^T} =  U^\star - (R^\star)^T (X^\star)^{-1} (R^\star).
$$
One may verify that 
$$
U^\star - (R^\star)^T (X^\star)^{-1} (R^\star) = 0,
$$
which gives that $v_k=0$ and the optimal controller is given by
\begin{align*}
u_k &= (R^\star)^T (X^\star)^{-1} x_k \\
	&= \begin{bmatrix}		
     		-283.7\\
     		-674.8
	\end{bmatrix}^T  (1473.1)^{-1} \\
	&=
\begin{bmatrix}		
     0.7908   & -2.5155
\end{bmatrix} x_k.
\end{align*}
The optimal value of the cost is $42.9116$. If we remove the last quadratic constraint with respect to $Q$,
we get the optimal cost $23.9361$, which is smaller than in the constrained case as expected.

\section{Conclusions}
We have considered the problem of multi-objective optimization of a system with multiplicative and 
additive noises, with arbitrary probability distributions. We showed that the problem can be converted 
to a semidefinite program where the optimization variable has the interpretation as the covariance matrix
$V$ of the vector $(x_k^T, ~u_k^T)^T$.

We showed that affine controllers are optimal and depend on the optimal solution of the covariance matrix
$V$. Furthermore, the optimal controllers are linear if all the quadratic forms in the constraints are convex in the controller.

We also showed that a necessary and sufficient condition for stabilizability is given by the solvability of a Lyapunov-like equation
in the covariance matrix $V$.

In this paper, we assumed that the additive noise $w_k$ is independent of the multiplicative noise $\sigma_k(i)$, for $i = 1, ..., M$.
Also, we have assumed that the matrix $B$ is deterministic.
However these assumptions may be removed. We could impose correlation between
$w$ and $\sigma(i)$ and also replace $B$ with 
$$B(\sigma) = B + \sum_i \sigma(i)B_i.$$ 
The proofs would follow readily, though giving more complicated expressions that
wouldn't shed more light on the main ideas of the paper.

\section{Acknowledgements}
The author would like to thank Prof. Bassam Bamieh for suggesting the problem.

\bibliography{../ref/mybib}


\end{document}